\documentclass[12pt,a4paper]{amsart}  
\usepackage{cmbright}
\usepackage[T1]{fontenc}

\usepackage{amsmath,amsthm,amssymb,enumerate,amsfonts,graphicx}
\usepackage{color}
\usepackage{epsfig}
\usepackage[numbers,sort&compress]{natbib}
\usepackage[T1]{fontenc}
\usepackage[utf8]{inputenc}
\usepackage{algorithm,algorithmic}
\usepackage{stmaryrd}
\usepackage{paralist}
\usepackage{tikz}
\tikzset{black node/.style={draw, circle, fill = black, minimum size = 5pt, inner sep = 0pt}}
\tikzset{normal/.style = {draw=none, fill = none, minimum size =0, rectangle}}

\usepackage{hyperref}
\hypersetup{
  pdftitle = {Notes on Tree- and Path-chromatic Number},
  pdfauthor = {Huynh, Reed, Wood, Yepremyan},
  colorlinks  
}

\usepackage{aliascnt}

\newtheorem{theorem}{Theorem}[section]

\newaliascnt{lemma}{theorem}
\newtheorem{lemma}[lemma]{Lemma}
\aliascntresetthe{lemma}

\newaliascnt{corollary}{theorem}

\aliascntresetthe{corollary}

\newaliascnt{conjecture}{theorem}
\newtheorem{conjecture}[conjecture]{Conjecture}
\aliascntresetthe{conjecture}

\newaliascnt{claim}{theorem}

\aliascntresetthe{claim}

\newaliascnt{open}{theorem}
\newtheorem{open}[open]{Open Problem}
\aliascntresetthe{open}

\renewcommand{\leq}{\leqslant}

\renewcommand{\geq}{\geqslant}

\newcommand{\treechi}{\mathsf{tree} \textnormal{-} \chi}
\newcommand{\pathchi}{\mathsf{path} \textnormal{-} \chi}

\begin{document}

\title{Notes on Tree- and Path-chromatic Number}

\author[T.~Huynh]{Tony Huynh}
\author[B.~Reed]{Bruce Reed}
\author[D.R.~Wood]{David R. Wood}
\author[L.~Yepremyan]{Liana Yepremyan}

\address[T.~Huynh and D.R.~Wood]{\newline School of Mathematics
\newline Monash University
\newline Melbourne, Australia}
\email{\{tony.bourbaki@gmail.com, david.wood@monash.edu\}}

\address[B.~Reed]{\newline School of Computer Science
\newline McGill University
\newline Montr\'eal, Canada}
\email{breed@cs.mcgill.ca}

\address[L.~Yepremyan]{\newline 
Department of Mathematics, Statistics, and Computer Science\newline
The University of Illinois at Chicago\newline
Chicago, USA \newline \newline
London School of Economics \newline 
Department of Mathematics \newline
London, UK}
\email{\{lyepre2@uic.edu, L.Yepremyan@lse.ac.uk\}}

\date{\today}
\sloppy

\begin{abstract}
\emph{Tree-chromatic number} is a chromatic version of treewidth, where the cost of a bag in a tree-decomposition is measured by its chromatic number rather than its size.  \emph{Path-chromatic number} is defined analogously. These parameters were introduced by Seymour~[JCTB~2016].  
In this paper, we survey all the known results on tree- and path-chromatic number and then present some new results and conjectures.  In particular, we propose a version of Hadwiger's Conjecture for tree-chromatic number.  As evidence that our conjecture may be more tractable than Hadwiger's Conjecture, we give a short proof that every $K_5$-minor-free graph has tree-chromatic number at most $4$, which avoids the Four Colour Theorem.  We also present some hardness results  and conjectures for computing tree- and path-chromatic number.
\end{abstract}

\maketitle

\section{Introduction}
\emph{Tree-chromatic number} is a hybrid of the graph parameters treewidth and chromatic number, recently introduced by Seymour~\cite{Seymour16}.  Here is the definition.

A \emph{tree-decomposition} of a graph $G$ is a pair $(T, \mathcal{B})$
where $T$ is a tree and $\mathcal{B}:=\{B_t \mid t \in V(T)\}$ is a collection of subsets of vertices of $G$, called \emph{bags}, satisfying:
\begin{itemize}
\item for each $uv \in E(G)$, there exists $t \in V(T)$ such that $u,v \in B_t$, and
\item for each $v \in V(G)$, the set of all $t \in V(T)$ such that $v \in B_t$ induces a non-empty subtree of $T$.  
\end{itemize}


A graph $G$ is \emph{$k$-colourable} if each vertex of $G$ can be assigned one of $k$ colours, such that adjacent vertices are assigned distinct colours. The \emph{chromatic number} of a graph $G$ is the minimum integer $k$ such that $G$ is $k$-colourable. 

For a tree-decomposition $(T,\mathcal{B})$ of $G$, the \emph{chromatic number} of $(T, \mathcal{B})$ is $\max \{\chi(G[B_t]) \mid t \in V(T)\}$.  The \emph{tree-chromatic number} of $G$, denoted $\treechi(G)$, is the minimum chromatic number taken over all tree-decompositions of $G$. The \emph{path-chromatic number} of $G$, denoted $\pathchi(G)$, is defined analogously, where we insist that $T$ is a path instead of an arbitrary tree. Henceforth, for a subset $B \subseteq V(G)$, we will abbreviate $\chi(G[B])$ by $\chi(B)$.  For $v \in V(G)$, let $N_G(v)$ be the set of neighbours of $v$ and $N_G[v] := N_G(v) \cup \{v\}$.  

The purpose of this paper is to  survey the known results on tree- and path-chromatic number, and to present some new results and conjectures.  

Clearly, $\treechi$ and $\pathchi$ are monotone under the subgraph relation, but unlike treewidth, they are not monotone under the minor relation.  For example, $\treechi(K_n)=n$, but the graph $G$ obtained by subdividing each edge of $K_n$ is bipartite and so $\treechi(G) \leq \chi(G)=2$. 

By definition, for every graph $G$, 
$$\treechi(G) \leq \pathchi(G) \leq \chi(G).$$
Section~\ref{sec:sep} reviews results that show that each of these inequalities can be strict and in fact, both of the pairs $(\treechi(G), \pathchi(G))$ and $(\pathchi(G), \chi(G))$ can be arbitrarily far apart.  

We present our new results and conjectures in Sections~\ref{sec:hadwiger}-\ref{sec:computing}. In Section~\ref{sec:hadwiger}, we propose a version of Hadwiger's Conjecture for tree-chromatic number and show how it is related to a `local' version of Hadwiger's Conjecture.  In Section~\ref{sec:k5}, we prove that $K_5$-minor-free graphs have tree-chromatic number at most $4$, without using the Four Colour Theorem.  We finish in Section~\ref{sec:computing}, by presenting some hardness results and conjectures for computing $\pathchi$ and $\treechi$.

\section{Separating $\chi$, $\pathchi$ and $\treechi$} \label{sec:sep}

Complete graphs are a class of graphs with unbounded tree-chromatic number.  Are there more interesting examples?  The following lemma of Seymour~\cite{Seymour16} leads to an answer.  A \emph{separation} $(A,B)$ of a graph $G$ is a pair of edge-disjoint subgraphs whose union is $G$.

\begin{lemma} \label{chromaticseparation}
For every graph $G$, there is a separation $(A,B)$ of $G$ such that $\chi(A \cap B) \leq \treechi(G)$ and
\[
\chi(A - V(B)), \chi(B - V(A)) \geq \chi(G)- \treechi(G).
\]
\end{lemma}

Seymour~\cite{Seymour16} noted that Lemma~\ref{chromaticseparation} shows that the random construction of Erd\H{o}s~\cite{erdos59} of graphs with large girth and large chromatic number also have large tree-chromatic number with high probability.  

Interestingly, it is unclear if the known \emph{explicit} constructions of large girth, large chromatic graphs also have large tree-chromatic number.  For example, \emph{shift graphs} are one of the classic constructions of triangle-free graphs with unbounded chromatic number, as first noted in~\cite{EH68}.  The vertices of the $n$-th shift graph $S_n$ are all intervals of the form $[a,b]$, where $a$ and $b$ are integers satisfying $1 \leq a < b \leq n$. Two intervals $[a,b]$ and $[c,d]$ are adjacent if and only if $b=c$ or $d=a$.  The following lemma (first noted in~\cite{Seymour16}) shows that the gap between $\chi$ and $\pathchi$ is unbounded on the class of shift graphs. 

\begin{lemma} \label{shift}
For all $n \in \mathbb N$, $\pathchi(S_n)=2$ and $\chi(S_n) \geq \lceil \log_2 n \rceil$.  
\end{lemma}

\begin{proof}
The fact that $\chi(S_n) \geq \lceil \log_2 n \rceil$ is well-known; we include the proof for completeness.  Let $\ell=\chi(S_n)$ and $\phi: V(S_n) \to [\ell]$ be a proper $\ell$-colouring of $S_n$.  For each $j \in [n]$ let $C_j=\{\phi([i,j]) \mid i < j\}$.  We claim that for all $j<k$, $C_j \neq C_k$.  By definition, $\phi([j,k]) \in C_k$.  If $C_j=C_k$, then $\phi([i,j])=\phi([j,k])$ for some $i<j$.  But this is a contradiction, since $[i,j]$ and $[j,k]$ are adjacent in $S_n$. Since there are $2^\ell$ subsets of $[\ell]$, $2^\ell \geq n$, as required.  

We now show that $\pathchi(S_n)=2$.  For each $i \in [n]$, let $B_i=\{[a,b] \in V(S_n) \mid a \leq i \leq b \}$.  Let $P_n$ be the path with vertex set $[n]$ (labelled in the obvious way).  We claim that $(P_n, \{B_i \mid i \in [n]\})$ is a path-decomposition of $S_n$.  First observe that $[a,b] \in B_i$ if and only if $a \leq i \leq b$. Next, for each edge $[a,b][b,c] \in E(S_n)$, $[a,b],[b,c] \in B_b$.  Finally, observe that for all $i \in [n]$, $X_i=\{[a,b] \in B_i \mid b=i\}$ and $Y_i=\{[a,b] \in B_i \mid b >i \}$ is a bipartition of $S_n[B_i]$.  Therefore, $S_n$ has path-chromatic number $2$, as required. 
\end{proof}

Given that shift graphs contain large complete bipartite subgraphs, the following question naturally arises. 

\begin{open}
Does there exist a function $f: \mathbb N \times \mathbb N \to \mathbb N$ such that for all $s \in \mathbb N$ and all $K_{s,s}$-free graphs $G$, $\chi(G) \leq f(s, \treechi(G) )$?
\end{open}

It is not obvious that the parameters $\pathchi$ and $\treechi$ are actually different.  Indeed, Seymour~\cite{Seymour16} asked if $\pathchi(G)=\treechi(G)$ for all graphs $G$?   Huynh and Kim~\cite{HK17} answered the question in the negative by exhibiting for each $k \in \mathbb{N}$, an infinite family of $k$-connected graphs for which $\treechi(G)+1=\pathchi(G)$. They also prove that the Mycielski graphs~\cite{Mycielski55} have unbounded path-chromatic number.

However, can $\treechi(G)$ and $\pathchi(G)$ be arbitrarily far apart?  Seymour~\cite{Seymour16} suggested the following family as a potential candidate. Let $T_n$ be the complete binary rooted tree with $2^n$ leaves. A path $P$ in $T_n$ is called a $\mathsf V$ if the vertex of $P$ closest to the root (which we call the \emph{low point} of the $\mathsf V$) is an internal vertex of $P$.  Let $G_n$ be the graph whose vertices are the $\mathsf V$s of $T_n$, where two $\mathsf V$s are adjacent if the low point of one is an endpoint of the other.  

\begin{lemma}[\cite{Seymour16}] 
\label{septreechi}
For all $n \in \mathbb N$, $\treechi(G_n)=2$ and 
$\chi(G_n)\geq \lceil \log_2 n \rceil$.
\end{lemma}

\begin{proof}
For each $t \in V(T_n)$, let $B_t$ be the set of $\mathsf V$s in $T_n$ which contain $t$.  We claim that $(T_n, \{B_t \mid t \in V(T_n)\})$ is a tree-decomposition of $G_n$ with chromatic number $2$.  First observe that if $P$ is a $\mathsf V$, then $\{t \in V(T_n) \mid P \in B_t\}=V(P)$, which induces a non-empty subtree of $T_n$.  Next, if $P_1$ and $P_2$ are adjacent $\mathsf V$s with $V(P_1) \cap V(P_2)=\{t\}$, then $P_1, P_2 \in B_t$.  Finally, for each $t \in B_t$, let $X_t$ be the elements of $B_t$ whose low point is $t$ and let $Y_t := B_t \setminus X_t$.  Then $(X_t, Y_t)$ is a bipartition of $G_n[B_t]$, implying that $\treechi(G_n)=2$. 

For the second claim, it is easy to see that $G_n$ contains a subgraph isomorphic to the $n$-th shift graph $S_n$.  Thus, $\chi(G_n) \geq \chi(S_n) \geq  \lceil \log_2 n \rceil$, by Lemma~\ref{shift}.
\end{proof}

Barrera-Cruz, Felsner,  M\'{e}sz\'{a}ros, Micek, Smith, Taylor, and Trotter~\cite{BFMMSTT19} subsequently proved that $\pathchi(G_n)=2$ for all $n \in \mathbb N$.  However, with a slight modification of the definition of $G_n$, they were able to
construct a family of graphs with tree-chromatic number 2 and unbounded path-chromatic number. 

\begin{theorem}[\cite{BFMMSTT19}] \label{septreepath}
For each integer $n \geq 2$, there exists a graph $H_n$ with $\treechi(H_n)=2$ and $\pathchi(H_n)=n$.  
\end{theorem}

The definition of $H_n$ is as follows.  A subtree of the complete binary tree $T_n$ is called a $\mathsf Y$ if it has three leaves and the vertex of the $\mathsf Y$ closest to the root of $T_n$ is one of its three leaves.  The vertices of $H_n$ are the $\mathsf V$s and $\mathsf Y$s of $T_n$.  Two $\mathsf V$s are adjacent if the low point of one is an endpoint of the other.  Two $\mathsf Y$s are adjacent if the lowest leaf of one is an upper leaf of the other.  A $\mathsf V$ is adjacent to a $\mathsf Y$ if the low point of the $\mathsf V$ is an upper leaf of the $\mathsf Y$.  The proof that $\pathchi(H_n)=n$ uses Ramsey theoretical methods for trees developed by Milliken~\cite{Milliken79}.

\section{Hadwiger's Conjecture for $\treechi$ and $\pathchi$} 
\label{sec:hadwiger}

 One could hope that difficult conjectures involving $\chi$ might become tractable for $\treechi$ or $\pathchi$, thereby providing insightful intermediate results.  Indeed, the original motivation for introducing $\treechi$ was a conjecture of Gy\'{a}rf\'{a}s~\cite{gyarfas85} from 1985, on $\chi$-boundedness of triangle-free graphs without long holes 
 \footnote{A \emph{hole} in a graph is an induced cycle of length at least $4$.}.
 
 \begin{conjecture}[Gy\'{a}rf\'{a}s's Conjecture~\cite{gyarfas85}] \label{conj:gyarfas}
 For every integer $\ell$, there exists $c$ such that every triangle-free graph with no hole of length greater than $\ell$ has chromatic number at most $c$.
 \end{conjecture} 
 
  Seymour~\cite{Seymour16} proved that Conjecture~\ref{conj:gyarfas} holds with $\chi$ replaced by $\treechi$.
 
 \begin{theorem}[\cite{Seymour16}] 
 \label{longhole1}
 For all integers $d \geq 1$ and $\ell \geq 4$, if $G$ is a graph with no hole of length greater than $\ell$ and $\chi(N_G(v)) \leq d$ for all $v \in V(G)$, then $\treechi(G) \leq d(\ell-2)$.  
 \end{theorem}
 
 Note that Theorem~\ref{longhole1} with $d=1$ implies that $\treechi(G) \leq \ell-2$ for every triangle-free graph $G$ with no hole of length greater than $\ell$. A proof of Gy\'{a}rf\'{a}s's Conjecture~\cite{gyarfas85} (among other results) was subsequently given by Chudnovsky, Scott, and Seymour~\cite{CSS17}.

The following is another famous conjectured upper bound on $\chi$, due to Hadwiger~\cite{Hadwiger43}; see \cite{SeymourHC} for a survey. 

\begin{conjecture}[\cite{Hadwiger43}]
If $G$ is a graph without a  $K_{t+1}$-minor, then $\chi(G) \leq t$.
\end{conjecture}

We propose the following weakenings of Hadwiger's Conjecture.

\begin{conjecture} \label{tree}
If $G$ is a graph without a $K_{t+1}$-minor, then $\treechi(G) \leq t$. 
\end{conjecture}

\begin{conjecture} \label{path}
If $G$ is a graph without a $K_{t+1}$-minor, then $\pathchi(G) \leq t$. 
\end{conjecture}

By Theorem~\ref{septreepath}, $\treechi(G)$ and $\pathchi(G)$ can be arbitrarily far apart, so Conjecture~\ref{tree} may be easier to prove than Conjecture~\ref{path}.  By Theorem~\ref{septreechi}, $\chi$ and $\treechi$ can be arbitrarily far apart, so Conjecture~\ref{tree} may be easier to prove than Hadwiger's Conjecture. We give further evidence of this in the next section, by proving Conjecture~\ref{tree} for $t=5$, without using the Four Colour Theorem. 

Robertson, Seymour, and Thomas~\cite{RST93} proved that every $K_6$-minor-free graph is $5$-colourable.  Their proof uses the Four Colour Theorem and is $83$ pages long. Thus, even if we are allowed to use the Four Colour Theorem, it would be interesting to find a short proof that every $K_6$-minor-free graph has tree-chromatic number at most $5$.

Conjectures~\ref{tree} and~\ref{path} are also related to a `local' version of Hadwiger's Conjecture via the following lemma.

\begin{lemma} 
\label{leafbag}
Let $(T, \{B_t \mid t \in V(T)\})$ be a $\treechi$-optimal tree-decomposition of $G$, with $|V(T)|$ minimal.  Then there are vertices $v \in V(G)$ and $\ell \in V(T)$ such that $N_G[v] \subseteq B_\ell$.  
\end{lemma}

\begin{proof}
Let $\ell$ be a leaf of $T$ and $u$ be the unique neighbour of $\ell$ in $T$.  If $B_\ell \subseteq B_u$, then $T-\ell$ contradicts the minimality of $T$.  Therefore, there is a vertex $v \in B_\ell$ such that $v \notin B_t$ for all $t \neq \ell$.  It follows that $N_G[v] \subseteq B_\ell$, as required.  
\end{proof}

Lemma~\ref{leafbag} immediately implies that the following `local version' of Hadwiger's Conjecture follows from Conjecture~\ref{tree}.  

\begin{conjecture} \label{local}
 If $G$ is a graph without a $K_{t+1}$-minor, then there exists $v \in V(G)$ such that $\chi(N_G[v]) \leq t$.
\end{conjecture}

It is even open whether Conjectures~\ref{tree},~\ref{path}, or~\ref{local} hold with an upper bound of $10^{100}t$ instead of $t$.  Finally, the following apparent weakening of Hadwiger's Conjecture (and strengthening of Conjecture~\ref{local}) is actually equivalent to Hadwiger's Conjecture.  

\begin{conjecture} \label{global}
 If $G$ is a graph without a $K_{t+1}$-minor, then $\chi(N_G[v]) \leq t$ for all $v \in V(G)$.  
\end{conjecture}

\begin{proof}[Proof of equivalence to Hadwiger's Conjecture]
Clearly, Hadwiger's Conjecture implies Conjecture~\ref{global}.  For the converse, let $G$ be a graph without a $K_{t+1}$-minor.  Let $G^+$ be the graph obtained from $G$ by adding a new vertex $v$
adjacent to all vertices of $G$.  Since $G^+$ has no $K_{t+2}$-minor, Conjecture~\ref{global} yields $\chi(N_{G^+}[v]) \leq t+1$.  Since $\chi(N_{G^+}[v])=\chi(G)+1$, we have $\chi(G) \leq t$, as required.  
\end{proof}

\section{$K_5$-minor-free graphs} \label{sec:k5}
As evidence that Conjecture~\ref{tree} may be more tractable than Hadwiger's Conjecture, we now prove it for $K_5$-minor-free graphs without using the Four Colour Theorem. We begin with the planar case.   

\begin{theorem} \label{planar}
For every planar graph $G$, $\treechi(G) \leq 4$.
\end{theorem}

\begin{proof}
We use the same tree-decomposition previously used by Eppstein~\cite{Eppstein99} and Dujmovi\'c, Morin, and Wood~\cite{DMW17}.

Say $G$ has  $n$ vertices.  We may assume that $n\geq 3$ and that $G$ is a plane triangulation. Let $F(G)$ be the set of faces of $G$. By Euler's formula, $|F(G)|=2n-4$ and $|E(G)|=3n-6$. Let $r$ be a vertex of $G$. Let $(V_0,V_1,\dots,V_t)$ be the bfs layering of $G$ starting from $r$. Let $T$ be a bfs tree of $G$ rooted at $r$. Let $T^*$ be the subgraph of the dual $G^*$ with vertex set $F(G)$, where two vertices are adjacent if the corresponding faces share an edge not in $T$. Thus $$|E(T^*)|=|E(G)|-|E(T)|=(3n-6)-(n-1)=2n-5=
|F(G)|-1=
|V(T^*)|-1.$$ 
By the Jordan Curve Theorem, $T^*$ is connected. Thus $T^*$ is a tree. 

For each vertex $u$ of $T^*$, if $u$ corresponds to the face $xyz$ of $G$, let $C_u := P_x\cup P_y\cup P_z$, where $P_v$ is the vertex set of the $vr$-path in $T$, for each $v\in V(G)$. See \cite{Eppstein99,DMW17} for a proof that $(T^*, \{C_u:u\in V(T^*)\})$ is a tree-decomposition of $G$. 

We now prove that $G[C_u]$ is 4-colourable. Let $\ell$ be the largest index such that $\{x,y,z\} \cap V_\ell \neq \emptyset$. 
For each $k \in \{0, \dots, \ell\}$, let $G_k=G[C_u \cap (\bigcup_{j=0}^k V_j)]$.  Note that $G_\ell=G[C_u]$.  We prove by induction on $k$ that $G_k$ is 4-colourable.  This clearly holds for $k \in \{0,1\}$, since $|V(G_1)| \leq 4$.  

For the inductive step, let $k \geq 2$.  For each $i \in \{0, \dots, \ell\}$, let $W_i=C_u \cap V_i$.  Since $W_i$ contains at most one vertex from each of $P_x, P_y$, and $P_z$, $|W_i| \leq 3$.  

First suppose $|W_i| \leq 2$ for all $i \leq k$. Since all edges of $G$ are between consecutive layers or within a layer, we can $4$-colour $G_k$ by using the colours $\{1,2\}$ on the even layers and $\{3,4\}$ on the odd layers.

Next suppose $|W_k| \leq 2$. We are done by the previous case unless $k=\ell, |W_\ell| \in \{1,2\}$, and $|W_{\ell-1}|=3$. By induction, let $\phi': V(G_{\ell-2}) \to [4]$ and $\phi: V(G_{\ell-1}) \to [4]$ be $4$-colourings of $G_{\ell-2}$ and $G_{\ell-1}$, respectively. If $|W_\ell| =1$, then clearly we can extend $\phi$ to a $4$-colouring of $G_\ell$.  So, we may assume  $|W_\ell|=2$.

Note that $\phi$ extends to a $4$-colouring of $G_\ell$ unless every vertex of $W_{\ell-1}$ is adjacent to every vertex of $W_\ell$ and the two vertices of $W_\ell$ are adjacent.  If $G[W_{\ell-1}]$ is a triangle, then $G[W_{\ell-1} \cup W_\ell]=K_5$, which contradicts planarity. If $G[W_{\ell-1}]$ is a path, say $abc$, then we obtain a $K_5$-minor in $G$ by contracting all but one edge of the $a$--$c$ path in $T$.  If $W_{\ell-1}$ is a stable set, then $\phi'$ can be extended to a $4$-colouring of $G_{\ell-1}$ such that all vertices in $W_{\ell-1}$ are the same colour.  This colouring can clearly be extended to a $4$-colouring of $G_\ell$. The remaining case is if $G[W_{\ell-1}]$ is an edge $ab$ together with an isolated vertex $c$. It suffices to show that there is a colouring of $G_{\ell-1}$ that uses at most two colours on $W_{\ell-1}$, since such a colouring can be extended to a $4$-colouring of $G_\ell$. Note that $\phi'$ can be extended to such a colouring unless $\phi'$ uses three colours on $W_{\ell-2}$ and $a$ and $b$ are adjacent to all vertices of $W_{\ell-2}$.  Since $\phi$ is a $4$-colouring, this implies that $\phi$ uses at most two colours on $W_{\ell-2}$.  Thus we may recolour $\phi$ so that only two colours are used on $W_{\ell-1}$, as required. 

Henceforth, we may assume $|W_k|=3$.  By induction, let $\phi: V(G_{k-1}) \to [4]$ be a $4$-colouring of $G_{k-1}$.  Let $\phi_{k-1}=\phi(W_{k-1})$.  

If $|\phi_{k-1}|=1$, then we can extend $\phi$ to a 4-colouring of $G_k$ by using $[4] \setminus \phi_{k-1}$ to 3-colour $W_k$.  

Suppose $|\phi_{k-1}|=2$.  By induction, $G_{k-2}$ has a 4-colouring $\phi'$. If $W_{k-1}$ is a stable set, then we can extend $\phi'$ to a $4$-colouring of $G_{k-1}$ such that all vertices of $W_{k-1}$ are the same colour.  Thus, $|\phi'_{k-1}|=1$, and we are done by the previous case.  Let $a,b \in W_{k-1}$ such that $ab \in E(G_{k-1})$. Let $c$ be the other vertex of $W_{k-1}$ (if it exists). By relabeling, we may assume that $\phi(a)=1, \phi(b)=2$, and $\phi(c)=2$.  Let $N(a)$ be the set of neighbours of $a$ in $W_k$ and $N(b,c)$ be the set of neighbours of $\{b,c\}$ in $W_k$.  Observe that $\phi$ extends to a $4$-colouring of $G_k$ unless $N(a)=N(b,c)=W_k$.  However, if, $N(a)=N(b,c)=W_k$, then we obtain a $K_5$-minor in $G$ by using $T$ to contract $W_k$ onto $\{x,y,z\}$ and $c$ onto $b$ (if $c$ exists).  This contradicts planarity.  

The remaining case is $|\phi_{k-1}|=3$. In this case, $\phi$ extends to a $4$-colouring of $G_k$, unless there exist distinct vertices $a,b \in W_{k-1}$ such that $a$ and $b$ are both adjacent to all vertices of $W_k$.  Again we obtain a $K_5$-minor in $G$ by using $T$ to contract $W_k$ onto $\{x,y,z\}$ and contracting all but one edge of the $a$--$b$ path in $T$.  
\end{proof}

We finish the proof by using Wagner's characterization of $K_5$-minor-free graphs~\cite{Wagner37}, which we now describe.  Let $G_1$ and $G_2$ be two graphs with $V(G_1) \cap V(G_2)=K$, where $K$ is a clique of size $k$ in both $G_1$ and $G_2$.  The \emph{$k$-sum} of $G_1$ and $G_2$ (along $K$) is the graph obtained by gluing $G_1$ and $G_2$ together along $K$ (and keeping all edges of $K$).  The \emph{Wagner graph} $V_8$ is the graph obtained from an $8$-cycle by adding an edge between each pair of antipodal vertices.  

\begin{theorem}[Wagner's Theorem~\cite{Wagner37}] \label{wagner}
Every edge-maximal $K_5$-minor-free graph can be obtained from $1$-, $2$-, and $3$-sums of planar graphs and $V_8$.
\end{theorem}

\begin{theorem}
For every $K_5$-minor-free graph $G$, $\treechi(G) \leq 4$.
\end{theorem}

\begin{proof}
Let $G$ be a $K_5$-minor-free graph.  We proceed by induction on $|V(G)|$.  We may assume that $G$ is edge-maximal.  First note that if $G=V_8$, then $\treechi(G) \leq \chi(G)=4$.  Next, if $G$ is planar, then $\treechi(G) \leq 4$ by Theorem~\ref{planar} (whose proof avoids the Four Colour Theorem).  By Theorem~\ref{wagner}, we may assume that $G$ is a $k$-sum of two graphs $G_1$ and $G_2$, for some $k \in [3]$.  Let $K$ be the clique in $V(G_1) \cap V(G_2)$ along which the $k$-sum is performed.  Since $G_1$ and $G_2$ are both $K_5$-minor-free graphs with $|V(G_1)|, |V(G_2)| < |V(G)|$, we have $\treechi(G_1)\leq 4$ and $\treechi(G_2) \leq 4$ by induction.  For $i \in [2]$, let $(T^i, \{B_t^i \mid t \in V(T^i)\})$ be a tree-decomposition of $G_i$ with chromatic number at most $4$.  Since $K$ is a clique in $G_i$, $K \subseteq B_x^1 \cap B_y^2$ for some $x \in V(T^1)$ and $y \in V(T^2)$.  Let $T$ be the tree obtained from the disjoint union of $T^1$ and $T^2$ by adding an edge between $x$ and $y$.  Then $(T, \{B_t^1 \mid t \in V(T^1)\} \cup \{B_t^2 \mid t \in V(T^2)\})$ is a tree-decomposition of $G$
with chromatic number at most $4$.
\end{proof}

\section{Computing $\treechi$ and $\pathchi$} \label{sec:computing}
We finish by showing some hardness results for computing $\treechi$ and $\pathchi$.  We need some preliminary results.
For a graph $G$, let $K_t^G$ be the graph consisting of $t$ disjoint copies of $G$ and all edges between distinct copies of $G$.  

\begin{lemma} \label{blowup}
For all $t \in \mathbb{N}$ and all graphs $G$ without isolated vertices, 
\[
(t-1)\chi(G)+2 \leq \treechi(K_t^G) \leq \pathchi(K_t^G) \leq t\, \chi(G).
\]
\end{lemma}

\begin{proof}
Let $(T, \{B_t \mid t \in V(T)\})$ be a $\treechi$-optimal tree-decomposition of $K:=K_t^G$, with $|V(T)|$ minimal. By Lemma~\ref{leafbag}, there exists $\ell \in V(T)$ and $v \in V(K)$ such that $N_K[v] \subseteq B_\ell$.  Since $G$ has no isolated vertices, $v$ has a neighbour in the same copy of $G$ in which it belongs.  Therefore, 
\[
\treechi(K) \geq \chi(B_\ell) \geq \chi(N_K[v]) \geq 2 +(t-1)\chi(G).
\]
For the other inequalities, $\treechi(K) \leq \pathchi(K) \leq \chi(K)=t\,\chi(G)$.
\end{proof}

We also require the following hardness result of Lund and Yannakakis~\cite{LY94}.

\begin{theorem}[\cite{LY94}] \label{chihard}
There exists $\epsilon >0$, such that it is NP-hard to correctly determine $\chi(G)$ within a multiplicative factor of $n^{\epsilon}$  for every $n$-vertex graph $G$.
\end{theorem}


Our first theorem is a hardness result for approximating $\treechi$ and $\pathchi$.

\begin{theorem} \label{hardapprox}
There exists $\epsilon' > 0$, such that it is NP-hard to correctly determine $\treechi(G)$ within a multiplicative factor of $n^{\epsilon'}$ for every $n$-vertex graph $G$.  The same hardness result holds for $\pathchi$ with the same $\epsilon'$.
\end{theorem}

\begin{proof}
We show the proof for $\treechi$. The proof for $\pathchi$ is identical.  Let $\epsilon'=\frac{\epsilon}{3}$, where $\epsilon$ is the constant from Theorem~\ref{chihard}. Let $G$ be an $n$-vertex graph.

Note that $K_n^G$ has $n^2$ vertices, and $(n^2)^{\epsilon'}=n^{\frac{2\epsilon}{3}}$.  If $k \in [\frac{\treechi(K_{n}^{G})}{n^{\frac{2\epsilon}{3}}}, n^{\frac{2\epsilon}{3}}\treechi(K_{n}^{G})]$, then 
$\frac{k}{n} \in [\frac{\chi(G)}{n^\epsilon}, n^\epsilon \chi(G)]$ by Lemma~\ref{blowup}.  Therefore, if we can approximate $\treechi(K_n^G)$ within a factor of $(n^2)^{\epsilon'}$, then we can approximate $\chi(G)$ within a factor of $n^\epsilon$.  
\end{proof}

For the decision problem, we use the following hardness result of Khanna, Linial, and Safra~\cite{KLS2000}.

\begin{theorem}[\cite{KLS2000}]
\label{chidecision}
Given an input graph $G$ with $\chi(G) \neq 4$, it is NP-complete to decide if $\chi(G) \leq 3$ or $\chi(G) \geq 5$.  
\end{theorem}

As a corollary of Theorem~\ref{chidecision}, we obtain the following.

\begin{theorem} \label{hardness}
It is NP-complete to decide if $\treechi(G) \leq 6$.  It is also NP-complete to decide if $\pathchi(G) \leq 6$.
\end{theorem}

\begin{proof}
Let $G$ be a graph without isolated vertices and $\chi(G) \neq 4$.  By Lemma~\ref{blowup}, if $\treechi(K_2^G) \leq 6$, then $\chi(G) \leq 3$ and if $\treechi(K_2^G) \geq 7$, then $\chi(G) \geq 5$.  Same for $\pathchi$. Finally, a tree- or path-decomposition and a $6$-colouring of each bag is a certificate that $\treechi(G) \leq 6$ or $\pathchi(G) \leq 6$.
\end{proof}

Combining the standard $O(2^n)$-time dynamic programming for computing pathwidth exactly (see Section 3 of \cite{SV2009}) and the $2^n n^{O(1)}$-time algorithm of Bj\"{o}rklund, Husfeldt, and Koivisto~\cite{BHK09} for deciding if $\chi(G) \leq k$, yields a $4^n n^{O(1)}$-time algorithm to decide to $\pathchi(G) \leq k$. As far as we know, there is no faster algorithm for deciding $\pathchi(G) \leq k$ (except for small values of $k$, where faster algorithms for deciding $k$-colourability can be used instead of~\cite{BHK09}).  

Finally, unlike for $\chi(G)$, we conjecture that it is still NP-complete to decide if $\treechi(G) \leq 2$.

\begin{conjecture}
It is NP-complete to decide if $\treechi(G) \leq 2$.  It is also NP-complete to decide if $\pathchi(G) \leq 2$.
\end{conjecture}

\bibliographystyle{abbrv}
\bibliography{references}

\end{document}